\documentclass[11pt]{article}

\usepackage{amssymb,amsmath}
\usepackage{amsthm}
\usepackage{hyperref}
\usepackage{mathrsfs}
\usepackage{textcomp}
\hypersetup{
    colorlinks,%
    citecolor=black,%
    filecolor=black,%
    linkcolor=black,%
    urlcolor=black
}

\usepackage{verbatim}

\usepackage{sectsty}

\makeatletter

\newdimen\bibspace
\renewenvironment{thebibliography}[1]{%
 \section*{\refname 
       \@mkboth{\MakeUppercase\refname}{\MakeUppercase\refname}}%
     \list{\@biblabel{\@arabic\c@enumiv}}%
          {\settowidth\labelwidth{\@biblabel{#1}}%
           \leftmargin\labelwidth
           \advance\leftmargin\labelsep
           \itemsep\bibspace
           \parsep\z@skip     %
           \@openbib@code
           \usecounter{enumiv}%
           \let\p@enumiv\@empty
           \renewcommand\theenumiv{\@arabic\c@enumiv}}%
     \sloppy\clubpenalty4000\widowpenalty4000%
     \sfcode`\.\@m}
    {\def\@noitemerr
      {\@latex@warning{Empty `thebibliography' environment}}%
     \endlist}

\makeatother

\makeatletter

\newtheorem{thm}{Theorem}[section]
\newtheorem{lem}{Lemma}[section]
\newtheorem{prop}{Proposition}[section]

\newtheorem{cor}{Corollary}[section]
\newtheorem{rem}{Remark}[section]

\def\XXint#1#2#3{{\setbox0=\hbox{$#1{#2#3}{\int}$}
  \vcenter{\hbox{$#2#3$}}\kern-.5\wd0}}

\newcommand{\al}{\alpha}                \newcommand{\lda}{\lambda}
\newcommand{\om}{\Omega}                \newcommand{\pa}{\partial}
\newcommand{\va}{\varepsilon}           \newcommand{\ud}{\mathrm{d}}
\newcommand{\be}{\begin{equation}}      \newcommand{\ee}{\end{equation}}

\newcommand{\R}{\mathbb{R}}

\begin{document}

\title{\textbf{Solutions of some Monge-Amp\`ere equations with \\isolated and line singularities}
\bigskip}

\author{\medskip Tianling Jin \  and \
Jingang Xiong}

\date{\today}

\maketitle

\begin{abstract}
In this paper, we study existence, regularity, classification, and asymptotical behaviors of solutions of some Monge-Amp\`ere equations with isolated and line singularities. We classify all solutions of $\det \nabla^2 u=1$ in $\R^n$ with one puncture point. This can be applied to characterize ellipsoids, in the same spirit of Serrin's overdetermined
problem for the Laplace operator. In the case of having $k$ non-removable singular points for $k>1$, modulo affine equivalence the set of all generalized solutions can be identified as an explicit orbifold of finite dimension. We also establish existence of global solutions with general singular sets, regularity properties, and optimal estimates of the second order derivatives of generalized solutions near the singularity consisting of a point or a straight line. The geometric motivation comes from singular semi-flat Calabi-Yau metrics.
\end{abstract}

\tableofcontents

\section{Introduction}

We say two Lebesgue measurable functions $u_1, u_2:\R^n\to \R$ are \emph{affine equivalent} if there exists
an $n\times n$ matrix $A$ with $\det A\neq 0$, $b=(b_1,\cdots, b_n)^{t}$ and a linear function $\ell(x)$ such that
$u_1(x)=(\det A)^{-2/n}u_2(Ax+b)-\ell(x)$ a.e. in $\R^n$; If $\det A=1$,  we say $u_1, u_2$ are \emph{unimodular affine equivalent}.

A celebrated theorem in the Monge-Amp\`ere equation theory asserts: \emph{Modulo the unimodular affine equivalence, $\frac{1}{2}|x|^2$ is the unique convex solution of}
\[
\det \nabla^2 u=1\quad\mbox{in }\R^n.
\]
The theorem was first proved by J\"orgens \cite{J} in dimension two using complex analysis methods.
An elementary and simpler proof, which also uses complex analysis, was later
given by Nitsche \cite{N}. J\"orgens' theorem was extended to smooth convex solutions in higher dimensions by Calabi \cite{C} for $n\le 5$ and by Pogorelov \cite{P} for all dimensions.  Another proof was given by Cheng and Yau \cite{CY}  along the lines of affine
geometry. Note that any local
generalized (or Alexandrov) solution of $\det\nabla^2u=1$ in dimension two is smooth, but this is
false in dimension $n\geq 3$. Caffarelli \cite{Caffarelli}
(see also Caffarelli-Li \cite{CL}) established J\"orgens-Calabi-Pogorelov
theorem for generalized solutions (or viscosity solutions). 
Trudinger-Wang \cite{TW2} proved that the only convex open subset $\om$ of $\R^n$ which admits a convex $C^2$ solution of $\det\nabla^2 u=1$ in $\om$ with $\lim_{x\to\pa\om}u(x)=\infty$ is $\om=\R^n$. Caffarelli-Li \cite{CL} established the asymptotical behaviors of viscosity solutions of $\det\nabla^2 u=1$ outside of a bounded convex subset of $\R^n$ for $n\ge 2$ (the case $n=2$ was studied before in Ferrer-Mart\'inez-Mil\'an \cite{FMM2, FMM1} using complex analysis), from which the J\"orgens-Calabi-Pogorelov theorem follows. Recently, we also gave a proof of J\"orgens' theorem (for $n=2$) in \cite{JX1} without using complex analysis, which allows us to obtain such Liouville type theorems for solutions of some degenerate Monge-Amp\`ere equations. 

In the language of affine differential geometry, the above theorem asserts that every convex improper affine hypersurface is an elliptic paraboloid. It is of interest to study affine hypersurfaces with singularities, from which part of this work is motivated.

In the paper \cite{J2},  J\"orgens showed that, modulo the unimodular affine equivalence, every smooth locally convex solution
of
\[
\det \nabla^2 u=1\quad\mbox{in }\R^2\setminus\{0\}
\]
has to be
\[
u_c=\int_0^{|x|}(\tau^2+c)^{\frac{1}{2}}\,\ud \tau
\]
for some $c\geq 0$.  One can check that $0$ is non-removable singular point of $u_c$ if and only if $c>0$.

In this paper, we would like to extend this J\"orgens' theorem to higher dimensions, explore the space of solutions in the case of containing multiple singular points, discuss the existence of global solutions with measure data, and study regularity properties and asymptotical behaviors of solutions of Dirichlet problems with isolated and line singularities. 

Recall that (see, e.g., \cite{G} and \cite{TW1}) for an open subset $\om$ of $\R^n$ and a Borel measure $\nu$ defined in $\om$, we say $u$ is a generalized solution, or Alexandrov solution, of the Monge-Amp\`ere equation
\[
\det \nabla^2 u=\nu \quad \mbox{in }\om,
\]
if $u$ is a locally convex function in $\om$ and the Monge-Amp\`ere measure associated with $u$ equals to $\nu$. Throughout the paper, we assume the dimension $n\ge 3$ without otherwise stated.

\begin{thm}\label{thm a}
Let $u$ be a  generalized solution of
\be\label{eq:global}
\det\nabla^2 u=1\quad\mbox{in }\R^n\setminus\{0\}.
\ee
Then $u$ is unimodular affine equivalent to
\[
\int_0^{|x|}(\tau^n+c)^{\frac{1}{n}}\,\ud \tau
\]
for some $c\geq 0$.
\end{thm}

From the proof of Theorem \ref{thm a}, $u$ in fact belongs to $C^{0,1}_{loc}(\R^n)$, and the constant $c=\frac{1}{\omega_n}|\pa u(0)|_{\mathcal{H}^n}$, where $\omega_n$ is the volume of $n$-dimensional unit ball, $\pa u(0)$ is the set  of the subgradients of $u$ at $0$ (see \cite{G}) and $|\cdot|_{\mathcal{H}^n}$ is the $n$-dimensional Lebesgue measure. Modulo the scaling $\bar{u}(x)=c^{-2/n}u(c^{1/n}x)$ for $c>0$, it follows from Theorem \ref{thm a} that in fact we have only two solutions of \eqref{eq:global}:
\[
\frac12 |x|^2 \quad \mbox{and} \quad \int_0^{|x|}(\tau^n+1)^{\frac{1}{n}}\,\ud \tau.
\]

Let us consider the case of $k$ puncture points for $k>1$. Let $u$ be a generalized solution of
\be\label{eq:many singularity1}
\det\nabla^2 u=1\quad \mbox{in } \R^n\setminus\{P_1,\cdots, P_k \}
\ee
for some distinct points $P_1,\cdots, P_k$ in $\R^n$. We will see from Proposition \ref{prop:across} in the next section that $u$ can be uniquely extended to be a convex function (still denoted as $u$) in $\R^n$, and thus $u$ is a generalized solution of
\be\label{eq:a1}
\det\nabla^2 u=1+\sum_{i=1}^k a_i \delta_{P_i}\quad \mbox{in } \R^n,
\ee
where $a_i=|\pa u(P_i)|_{\mathcal{H}^n}$, $\delta_{P_i}$ is the delta measure centered at $P_i$, $1\le i\le k$. We say that $P_i$ is a non-removable singular point of \eqref{eq:many singularity1} or \eqref{eq:a1} if $|\pa u(P_i)|_{\mathcal{H}^n}\neq 0$. If $|\pa u(P_i)|_{\mathcal{H}^n}= 0$, then $P_i$ is a removable singular point in the Alexandrov sense. For removable singularities of classical solutions of Monge-Amp\`ere equations, we refer to \cite{J2, B, SW}.

\begin{thm} \label{them:elementary1} Modulo the affine equivalence, the set of all generalized solutions of \eqref{eq:many singularity1} with $k$ distinct non-removable singular points can be identified as an orbifold of dimension $d_{n,k}$ , where
\be\label{eq:orbifold dim} 
d_{n,k}=
\begin{cases}
\frac{(k-1)(k+2)}{2},\quad\mbox{if}\quad k-1\le n,\\
(k-1)(n+1)-\frac{n(n-1)}{2},\quad\mbox{if}\quad k-1> n.
\end{cases}
\ee
Moreover, when $n=3$ or $4$, every generalized solutions of \eqref{eq:many singularity1} is smooth away from the set of line segments each of which connects two singular points.
\end{thm}

The orbifold in Theorem \ref{them:elementary1} will be given explicitly in the proof of Corollary \ref{cor:multip point}. When $n=2$, Theorem \ref{them:elementary1} was proved by G\'alvez, Mart\'inez and Mira \cite{GMM} using one complex variable methods, and it follows from two dimensional Monge-Amp\`ere equation theory that the solutions are smooth away from the set of the singular points. In general, we know from \cite{CL} that every solution of \eqref{eq:many singularity1} for $n\ge 3$ is strictly convex (and thus, smooth) outside the convex hull of $\{P_1,\cdots, P_k\}$. 

We are also interested in seeking global solutions of Monge-Amp\`ere equation with more general singular sets than isolated points.
The existence of such solutions follows from the next theorem, which shows existence of global solutions of Monge-Amp\`ere equations with measure data.  In the rest of the paper, we denote $\mathcal{A}$ as the set of real $n\times n$ positive definite matrices with determinant $1$, and $B_r$ as the ball in $\R^n$ centered at $0$ of radius $r$. 

\begin{thm}\label{existenceY}
Let $\mu$ be a locally finite Borel measure such that the support of $(\mu-1)$ is bounded. Then for every $c\in \R, b\in\R^n, A\in\mathcal{A}$, there exists a unique convex Alexandrov solution of
\be\label{eq:1+mu}
\det \nabla^2 u=\mu\quad\mbox{in }\R^n
\ee
satisfying
\[\lim_{|x|\to +\infty}|E(x)|=0,\]
where $E(x)=u(x)-(\frac 12 x^T Ax+b\cdot x+c)$.
\end{thm}

If $\ud\mu=f(x)\ud x$ for some $f\in C(\R^n)$ satisfying supp$(f-1)$ is bounded and $\inf_{\R^n} f>0$, then Theorem \ref{existenceY}  was proved in \cite{CL}. We also have decay rates of $E$ and all of its derivatives in Theorem \ref{existenceY} and Theorem \ref{existence1} (in Section \ref{sec:measure data}), which follows from \cite{CL}.

The following theorem discusses some strictly convex properties of solutions of Monge-Amp\`ere equations with singularities, from which the regularity part of Theorem \ref{them:elementary1} follows. For a subset $\Gamma$ of $\Omega\subset\R^n$, we denote $\Gamma\subset\subset\om$ if $\Gamma\subset\overline\Gamma\subset \om$.

\begin{thm} \label{thm:strict convexity}
Let $\om$ be a bounded convex domain of $\R^n$, $0<\lambda\le\Lambda<\infty$, $\varphi\in C^{1,\beta}(\pa \om)$ for some $\beta>1-\frac 2n$ and $\Gamma \subset\subset\om$.  Let $u\in C(\overline \om)$ be a generalized convex solution of the Dirichlet problem
\[
\begin{split}
\begin{aligned}
\lambda\le\det &\nabla^2 u\le \Lambda&\quad& \mbox{in }\om\setminus\Gamma,\\
u&=\varphi& \quad & \mbox{on } \pa \om.
\end{aligned}
\end{split}
\]
Then $u$ is locally strictly convex in $\om\setminus\mathcal{C}(\Gamma)$, where $\mathcal{C}(\Gamma)$ is the convex hull of $\Gamma$. Moreover, when $n=3$ or $4$, and $\Gamma$ is a set consisting of finitely many points and line segments, then $u$ is locally strictly convex in $\om\setminus\mathcal{L}(\Gamma)$, where $\mathcal{L}(\Gamma)$ is the set of line segments each of which connects two points of $\Gamma$.
\end{thm}

Some strengthened strict convexity results will be discussed in Section \ref{sec:Regularity of local solutions}. 

Next, we move to discuss asymptotical behaviors of solutions of Monge-Amp\`ere equations with isolated and line singularity in bounded domains. 
\begin{thm} \label{thm:asymptotical behavior}
Let $\om$ be a bounded convex domain of $\R^n$ with $n\ge 2$, $\Gamma\subset\subset\om$ be either a point or a straight line segment.  Let $u$ be a convex function in $\om$ and $u\in C^2(\om\setminus\Gamma)$ satisfying
\be \label{eq:mae-singulary-line}
\begin{split}
\begin{aligned}
\det \nabla^2 u&=1& \mbox{in }\om\setminus\Gamma.\\
\end{aligned}
\end{split}
\ee
Then
\be\label{eq:mae-sigular-line-estimate}
|\nabla^2 u(x)|\leq \frac{C}{dist(x,\Gamma)},
\ee
where $C>0$ is independent of $x$.
\end{thm}

We remark that the rate $O(1/dist (x, \Gamma))$ in \eqref{eq:mae-sigular-line-estimate} is the best we can have since the solution in Theorem \ref{thm a} is indeed of this rate, and an application of \eqref{eq:mae-sigular-line-estimate} can be found in Corollary \ref{cor:metric compl}. The assumption on the regularity on $u$ will be satisfied if some mild boundary condition is given as in Theorem \ref{thm:strict convexity}. Our proof also works for general set $\Gamma$ other than a point or a straight line with the estimate \eqref{eq:mae-sigular-line-estimate} replaced by $C/dist(x,\mathcal{C}(\Gamma))$ (see Remark \ref{rem:convex set}). Some explicit dependence of the constant $C$ will be given in Theorem \ref{thm:c} and Theorem \ref{thm:line}.

We shall end the introduction with an application of Theorem \ref{thm a}. In \cite{S}, Serrin proved that whenever $\Omega$ is a bounded smooth domain in $\R^n$, and $\nu$ is the outer normal of $\partial \Omega$, if $u\in C^2(\overline \Omega)$ is a solution of
\be\label{eq:serrin}
\begin{cases}
\begin{aligned}
\Delta u=n &\quad\mbox{in }\Omega,\\
u=0 &\quad\mbox{on }\partial\Omega,\\
\pa u/\pa\nu=1 &\quad\mbox{on }\partial\Omega,
\end{aligned}
\end{cases}
\ee
then after some translation $\Omega$ has to the unit ball and $u=\frac{|x|^2-1}{2}$. The proof of Serrin used the method of moving planes. Later,  Brandolini, Nitsch, Salani and Trombetti \cite{BNST} extended Serrin's result to $\sigma_k(\nabla^2 u)$, the $k$-th elementary symmetric function of $\nabla^2 u$, via an alternative approach. Namely, they showed that whenever $\Omega$ is a bounded smooth domain, and $\nu$ is the outer normal of $\partial \Omega$, if $u\in C^2(\overline \Omega)$ is a solution of
\[
\begin{cases}
\begin{aligned}
\sigma_k(\nabla^2 u)&= \binom{n}{k} &\quad&\mbox{in }\Omega,\\
u&=0 &\quad&\mbox{on }\partial\Omega,\\
\pa u/\pa\nu&=1 &\quad&\mbox{on }\partial\Omega
\end{aligned}
\end{cases}
\]
with $k=1,2,\cdots, n$, then after some translation $\Omega$ has to the unit ball and $u=\frac{|x|^2-1}{2}$. In this spirit, we show that

\begin{thm}\label{thm:e}
Let $\Omega$ be a bounded smooth domain in $\R^n$ with $n\ge 2$. If there exists a locally convex function $u\in C^1(\R^n\setminus \Omega)\cap C^2(\R^n\setminus \overline\Omega) $ satisfying
\[
\begin{cases}
\begin{aligned}
\det\nabla^2 u=1 &\quad\mbox{in }\R^n\setminus\overline\Omega,\\
u=0 &\quad\mbox{on }\partial\Omega,\\
\pa u/\pa\nu=0 &\quad\mbox{on }\partial\Omega,
\end{aligned}
\end{cases}
\]
where $\nu$ is the unit outer normal of $\pa\om$, then $\Omega$ has to be an ellipsoid.
\end{thm}

As mentioned in the recent paper Shahgholian \cite{Sh} that little is known about \eqref{eq:serrin} in unbounded domains even with quadratic growth condition on $u$ near infinity.
We refer to \cite{Sh} and references therein for more discussions and open problems in this direction. It is also interesting to ask similar questions for $\sigma_k(\nabla^2 u)$ instead of $\det\nabla^2 u$ in Theorem \ref{thm:e}.

This paper is organized as follows. In Section \ref{sec2}, we consider the case of one singularity and prove Theorems \ref{thm a} and \ref{thm:e}. In Section \ref{sec3}, we study the case of multiple singularities and show Theorems \ref{them:elementary1} and \ref{thm:strict convexity}. Section \ref{sec:line and Y} is devoted to the case of line singularity and proving Theorems \ref{existenceY} and \ref{thm:asymptotical behavior}.
\bigskip

\noindent\textbf{Acknowledgements:} Both authors thank Professor YanYan Li
for valuable suggestions and constant encouragement. We also thank Tian Yang for discussions.
The second author is grateful to Professor Gang Tian for his suggestions to study Monge-Amp\`ere equations with singularities. He was supported in part by the First Class Postdoctoral Science Foundation of China (No. 2012M520002).

\section{One singular point}\label{sec2}

\subsection{Classification of global solutions and an application}
\begin{prop}\label{prop:across}
Let $u$ be a locally convex function in $B_1\setminus \{0\}$, where $B_1$ is the unit ball in $\R^n$ with $n\ge 2$. Then $u$ can be extended to be a convex function in $B_1$.
\end{prop}
\begin{proof} This proposition should be known, but we still provide a proof here for completeness.

Step 1: We show that $\limsup_{|x|\to 0}u(x)<\infty$. For any $x$ close to $0$, we can choose $x_1, x_2$ away from $0$ and $\pa B_1$ such that $2x=x_1+x_2$. Since $u$ is locally convex in $B_1\setminus \{0\}$, $u$ is convex on the line segment from $x_1$ to $x_2$, i.e.,
\[
u(x)\leq \frac{u(x_1)+u(x_2)}{2}.
\]
Since $u$ is convex near $x_1$ and $x_2$,  $u$ is continuous and hence bounded near $x_1$ and $x_2$. Thus, we have $\limsup_{|x|\to 0}u(x)<\infty$.

\medskip

Step 2: Define $u(0)= \limsup_{|x|\to 0}u(x)<\infty$. We will show that for any $x, y\in B_1$ and $\lambda\in [0,1]$,
\[
u(\lambda x+(1-\lambda)y)\leq \lambda u(x)+(1-\lambda)u(y).
\]
We only need to show the above inequality when $0$ is on the line segment from $x$ to $y$. If $x\neq 0$ and $y\neq 0$, we can choose $x_i, y_i, z_i\in B_1$ such that $x_i\to x,\ y_i\to y, z_i\to 0\ \mbox{as}\  i\to\infty$,
\[
\lim_{z_i\to 0}u(z_i)=u(0),
\]
\[
z_i=\lambda x_i+(1-\lambda)y_i,
\]
and $0$ is not on the line segment from $x_i$ to $y_i$ for each $i$.
Since
\[
u(z_i)\leq \lambda u(x_i)+(1-\lambda)u(y_i),
\]
we have
\[
u(0)\leq \lambda u(x)+(1-\lambda)u(y)
\]
 by sending $i\to\infty$.

If $x=0$, and $y\neq 0$, we choose $x_i\to 0, \ y_i\to y$ as $i\to\infty$ such that $0$ is not on the line segment from $x_i$ to $y_i$ for each $i$. For every $\lambda\in [0,1)$, we have
\[
u(\lda x_i+(1-\lambda)y_i)\leq \lambda u(x_i)+(1-\lambda)u(y_i).
\]
Since $u$ is continuous near $y$ and $(1-\lambda)y$, we have
\[
u((1-\lambda)y)\leq\lambda \limsup_{i\to\infty}u(x_i)+(1-\lambda)u(y) \leq \lambda u(0)+(1-\lambda)u(y).
\]
by sending $i\to\infty$.

Therefore, we can conclude that $u$ is convex in $B_1$ from the fact that a locally convex function in a convex domain is convex. In particular, $u$ is continuous in $B_1$.
\end{proof}
\begin{prop}\label{prop:lineconvex}
Let $\Gamma\subset\subset B_1\subset\R^n$ be a straight line segment with $n\ge 3$. Let $u$ be locally convex in $B_1\setminus\Gamma$. Then $u$ can be uniquely extended to be a convex function in $B_1$.
\end{prop}

\begin{proof}
Our proof works for $\Gamma$ to be an open, or closed, or half open half closed line segment. Without loss of generality, we may assume that 
\[
\Gamma=\{x=(x_1, \cdots, x_n)\in B_1: x_1\in [-1/2, 1/2], x_j=0\mbox{ for } 2\le j\le n\}.
\]
Let $H_s=\{x\in B_1: x_1=s\}$. Then $u(s,\cdot)$ is a locally convex function in $H_s\setminus\Gamma$. By Proposition \ref{prop:across}, $u(s,\cdot)$ can be extended to be a convex function, which is still denoted as $u(s,\cdot)$, on $H_s$. Moreover, it is clear that $u(\cdot, 0)$ is convex on $\Gamma$. We will show that $u$ is convex on any line segment $\tilde\Gamma$ in $B_1$. Since $n\ge 3$, by approximations as in the proof of Proposition \ref{prop:across}, we only need to show that for any line segment $\tilde\Gamma$ satisfying $\tilde\Gamma\cap\Gamma=P_s=(s,0,\cdots,0)$,
\be\label{eq:continuous}
\lim_{x\in\tilde\Gamma, x\to P_s}u(x)=u(P_s).
\ee
Suppose first that $\tilde\Gamma$ does not lie on the $x_1$-axis. For $x\in \tilde\Gamma$, let $x'$ be the projection point of $x$ from $\tilde\Gamma$ to $H_s$. Then
\[
\begin{split}
|u(P_s)-u(x)|&\le |u(P_s)-u(x')|+|u(x')-u(x)|\\
&\le |u(P_s)-u(x')|+C(u)|x'-x|,
\end{split}
\]
where $C(u)=\sup_{B_{0.8}\setminus B_{0.6}}|\nabla u|<\infty$. Since $u$ in continuous on $H_s$, \eqref{eq:continuous} holds.

If $\tilde\Gamma$ lies on the $x_1$-axis, for $x\in\tilde\Gamma$, we choose $x'\in H_s$ such that $|x-P_s|=|x'-P_s|$. Then \eqref{eq:continuous} follows in the same way as above.
\end{proof}

\begin{cor}\label{cor:Yconvex}
Let $\Gamma\subset \subset B_1\subset \R^n$ be a union of finite many line segments, where $n\ge 3$. Let $u$ be locally convex in $B_1\setminus\Gamma$. Then $u$ can be uniquely extended to be a convex function in $B_1$.
\end{cor}

\begin{proof}
By Proposition \ref{prop:lineconvex}, $u$ can be extended to be a locally convex function in $B_1\setminus \Gamma_p$, where $\Gamma_p$ is set of finitely many points in $B_1$. Then Corollary \ref{cor:Yconvex} follows from Proposition \ref{prop:across}.
\end{proof}

\begin{prop}\label{prop:sub}
Suppose $u\in C(B_1)$ is locally convex in $B_1\setminus\{0\}\subset \R^n$ with $n\ge 2$, and is a generalized solution of
\[
\det\nabla^2 u=1\quad\mbox{in }B_1\setminus\{0\},
\]
then
\[
\det\nabla^2 u= 1+|\pa u(0)|_{\mathcal{H}^n}\delta_0\quad\mbox{in }B_1.
\]
\end{prop}

\begin{proof}
The proposition follows directly from Proposition \ref{prop:across} and the definition of generalized solutions.
\end{proof}

Clearly, Proposition \ref{prop:sub} still holds when $1$ is replaced by any nonnegative bounded function.

Next, we recall the asymptotical behaviors of solutions of $\det \nabla^2 u=1$ in exterior domains near the infinity established in \cite{CL},
which will play a crucial role in our proofs.

\begin{thm}[Corollary 1.3 in \cite{CL}] \label{lem:*}Let $O$ be a bounded open convex subset of $\R^n$, and let $u\in C(\R^n\setminus \overline O)$ be
a generalized solution of
\[
\det \nabla^2 u=1\quad \mbox{in }\R^n\setminus \overline O.
\]
Then $u\in C^\infty(\R^n\setminus \overline O)$, and we have the following:
\begin{itemize}
\item[(i)] For $n\geq 3$, there exists some linear function $\ell(x)$, and $A\in\mathcal{A}$ such that
\[
\limsup_{|x|\to\infty}|x|^{n-2}|u(x)-(\frac 12 x^TAx+\ell(x))|<\infty,
\]
where $\mathcal{A}$ is the set of real $n\times n$ positive definite matrices with determinant $1$.

\item[(ii)] For $n=2$,  there exists some linear function $\ell(x)$, $d\in \R$, and $A\in\mathcal{A}$ such that
\[
\limsup_{|x|\to\infty}|x||u(x)-(\frac 12 x^TAx+d\log\sqrt{x^TAx}+\ell(x))|<\infty.
\]
\end{itemize}
\end{thm}

With the help of Theorem \ref{lem:*}, we are ready to show Theorem \ref{thm a}.
\begin{proof}[Proof of Theorem \ref{thm a}] We shall also show the case $n=2$, which provides another proof of a theorem of J\"orgens in \cite{J2} mentioned in the introduction.

 Case 1: $n\geq 3$.
By Theorem \ref{lem:*}, $u$ is smooth in $\R^n\setminus\{0\}$, and after a suitable affine transformation and subtracting a linear function we can assume that
\[
\limsup_{|x|\to\infty}|x|^{n-2}|u(x)-\frac 12 |x|^2|<\infty.
\]
By Proposition \ref{prop:across} and Proposition \ref{prop:sub}, $u$ is convex in $\R^n$ and satisfies
\be\label{eq:global singular}
\begin{split}
\det \nabla^2 u=1+|\pa u(0)|_{\mathcal{H}^n}\delta_0 \quad \mbox{in }\R^n
\end{split}
\ee
in Alexandrov sense.
By the comparison principle (see \cite{G}), we have
\[
u(x)\le \frac 12 |x|^2 \quad\mbox{in }\R^n.
\]
In particular, $u(0)\le 0$. Hence we can choose $c\ge 0$ such that
\[
\limsup_{|x|\to\infty}|x|^{n-2}|v(x)-\frac 12 |x|^2|<\infty,
\]
where
\[
v(x):=\int_0^{|x|}(\tau^n+ c)^{\frac{1}{n}}\,\ud \tau+u(0).
\]
Thus,
\[
\det\nabla^2 u=\det\nabla^2 v=1\quad\mbox{in }\R^n\setminus\{0\},
\]
\[
v(0)=u(0)\quad\mbox{and}\quad\limsup_{|x|\to\infty}|x|^{n-2}|u(x)-v(x)|<\infty.
\]
It follows that for some positive definite matrix function $(a_{ij}(x))$,
\[
a_{ij}\nabla_{ij}w=0\ \mbox{in}\ \R^n\setminus\{0\},\ w(0)=0,\ \mbox{and}\ \limsup_{|x|\to\infty}|w(x)|=0,
\]
where $w=u-v$. By the maximum principle, $w\equiv 0$, i.e., $u\equiv v$.

Case 2: $n=2$. For simplicity, we let $a=|\pa u(0)|_{\mathcal{H}^n}$. We may assume that
\be\label{asym beha 2}
\limsup_{|x|\to\infty}|x||u(x)-\frac 12 |x|^2-d\log|x||<\infty
\ee
for some $d$. It follows from the proof of (1.9) in \cite{CL} that $d=\frac{a}{2\pi}$ (see also \eqref{value of d} in the next section). Let
\[
w(x):=\int_0^{|x|}\sqrt{s^2+a/\pi}\ud s.
\]
It is clear that $w$ satisfies \eqref{eq:global singular} and \eqref{asym beha 2}. By the comparison principle, $u\equiv w$.

Indeed, the proof given in Case 2 can also be applied to show Case 1.
\end{proof}

Theorem \ref{thm:e} is a consequence of Theorem \ref{thm a}.
\begin{proof}[Proof of Theorem \ref{thm:e}] It is clear that $u$ locally strictly convex in $\R^n\setminus\overline\om$. Consequently, since $u=\pa u/\pa\nu=0$ on $\pa\om$, we hace $u>0$ in $\R^n\setminus\overline\om$. Thus, if we extend $u$ to be identically zero (which is still denoted as $u$) in $\om$, then $u$ is convex in $\R^n$ and hence $\om$ is convex. Let $u^*$ be the Legendre transformation of $u$ given by $u^*(y)=\sup\{x\cdot y-u(x): x\in\R^n\}$ for $y\in\R^n$. Then $u^*$ is $C^2$ and locally convex in $\Omega^*:=\pa u(\R^n\setminus\overline\om)$,
and satisfies
$\det \nabla^2 u^*=1$ in $\Omega^*$. We claim that $\Omega^*=\R^n\setminus \{0\}$. Indeed, since $u$ is locally strictly convex in $\R^n\setminus\overline\om$ and $\nabla u=0$ on $\pa\om$, we have that $0\not\in\om^*$. Secondly, for every $p\in \R^n\setminus\{0\}$, we can choose a positive constant $C$ large enough such that $u(x)>l(x):=p\cdot x - C$. This is can be done because of the asymptotical behavior of $u$ in Theorem \ref{lem:*}. Decrease $C$ such that $u(x)$ touches $l(x)$ at $x^*$. Since $p\neq 0$, $x^*\in \R^n\setminus\overline\om$. The claim is proved. By Theorem \ref{thm a} for $n\ge 3$ and the theorem of J\"orgens in \cite{J2} for $n=2$, $\overline\om=\pa u^*(0)$ is an ellipsoid.
\end{proof}

\subsection{Asymptotical behaviors of solutions near the isolated singularity}\label{sec4}

\begin{thm} \label{thm:c}
Let $\om$ be a bounded strictly convex domain in $\R^n$ with $\pa \om \in C^4$ and $n\ge 2$, $f\in C^{1,1}(\overline \om)$, $f>0$ in $\overline \om$ and $\varphi\in C^4(\pa \om)$. For any
$x_0\in \om$ and $a>0$,  let $u\in C(\overline \om)$ be the unique generalized convex solution of the Dirichlet problem
\be \label{mae1}
\begin{split}
\begin{aligned}
\det \nabla^2 u&=f+a \delta_{x_0}&\quad& \mbox{in }\om,\\
u&=\varphi& \quad & \mbox{on } \pa \om.
\end{aligned}
\end{split}
\ee
Then $u\in C^{0,1}(\om) \cap C^{3,\al}_{loc}(\overline \om \setminus \{x_0\})$ for any $\al\in (0,1)$. Moreover, we have
\be\label{important esti}
|\nabla^2 u(x)|\leq \frac{C}{|x-x_0|},
\ee
where $C>0$ depends only on $\om, n, a, \min_{\overline \om}f,\|f\|_{C^{1,1}(\overline \om)}, \|\varphi\|_{C^4(\pa \om)}$  and $ dist(x_0,\pa \om)$.
\end{thm}

The proof of Theorem \ref{thm:c} will be postponed to be shown in Section \ref{asymp-line} (see Theorem \ref{thm:line}). We remark that $\pa u(x_0)$ is indeed a compact convex set in $\R^n$. Moreover, the Lebesgue measure $|\pa u(x_0)|_{\mathcal{H}^n}=a$. Since $a>0$, $u$ has a tangent cone at $x_0$ whose level sets are convex.  In \cite{Savin}, Savin proved that those level sets are $C^{1,1}$ regular.

\begin{cor}\label{cor:metric compl}
Let $u$ be a smooth convex solution of $\det \nabla^2 u=1$ in $B_2\setminus \{0\}\subset \R^n$ with $n\ge 2$, and $g=u_{ij}\ud x_i\ud x_j$ be a Riemannian metric. Then
the completion of $(\overline{B_1}\setminus \{0\}, d_g)$ is equal to $\overline{B_1}$, where $d_g$ is the induced distance of $g$.
\end{cor}

\begin{proof}
It follows from \eqref{important esti} that for every $P\in B_{1}$ the length of the line segment connecting $P$ and $0$ under $g$ is less than $C\int_{0}^1x^{-1/2}\ud x$, which is finite. Thus, $d_g$ can be defined in the whole ball $\overline B_1$. Lastly, it is elementary to check that the extended distance function $d_g$ satisfies the triangle inequality.
\end{proof}

The geometric motivations of Theorem \ref{thm:c} and Corollary \ref{cor:metric compl} can be found in Gross-Wilson \cite{GrW}, Loftin \cite{L1}, Loftin-Yau-Zaslow \cite{LYZ} and references therein.

\section{Multiple isolated singularities}\label{sec3}

\subsection{The space of global solutions}

In this section, we study solutions of Monge-Amp\`ere equations with multiple isolated singularities. The solution counting result in Theorem \ref{them:elementary1} is restated in a more explicit way in Corollary \ref{cor:multip point}.

\begin{prop}\label{prop:to the main result}
Let $k\ge 1$. For any given positive numbers $a_1, \cdots, a_k$ and points $P_1,\cdots, P_k$ in $\R^n$,
there exists a unique generalized solution $u$ of 
\be\label{eq:more pts}
\det \nabla^2 u=1+\sum_{i=1}^k a_i \delta_{P_i}\quad \mbox{in }\R^n
\ee
 with prescribed asymptotical behavior
\be\label{eq:prescrib behavior}
\limsup_{|x|\to \infty}|x|^{n-2}|u(x)-\frac12 |x|^2|<\infty.
\ee
\end{prop}

\begin{proof} The uniqueness is clear, and we shall show the existence. The existence can actually follow from Theorem \ref{existenceY}, but we would like to provide a simpler proof for this particular case.

(i) We claim that there exists a unique solution $u_i$ of
\be\label{eq:i}
\det \nabla^2 u_i=1+k^n a_i \delta_{P_i} \quad \mbox{in }\R^n,
\ee
with the asymptotical behavior \eqref{eq:prescrib behavior}. We only need to show the existence. As in the previous section, one can find a radial symmetric (w.r.t. $P_i$) solution $v_i(x)=v_i(|x-P_i|)$ of \eqref{eq:i} satisfying
\[
\limsup_{|x|\to \infty}|x|^{n-2}|v_i(x)-\frac12 |x-P_i|^2|<\infty.
\]
Then $u_i=v_i+P_i\cdot x-\frac12 |P_i|^2$ is a desired solution.

(ii) Let $\underline{u}=\frac{1}{k}\sum_{i=1}^k u_i$. It is clear that $\underline{u}$ satisfies \eqref{eq:prescrib behavior}.
If $x\neq P_i$ for all $i$, then
\[
(\det\nabla^2 \underline{u}(x))^{\frac{1}{n}}\geq \frac{1}{k}\sum_{i=1}^k  (\det\nabla^2 u_i(x))^{\frac{1}{n}} =1.
\]
Hence $\det\nabla^2 \underline{u}(x) \geq 1 $ in $\R^n$. For any Borel set $E\subset \R^n$, let $I:=\{i: P_i\in E\}$. It follows that,  for sufficiently small $\va>0$,
\[
\begin{split}
|\pa \underline{u}(E)|_{\mathcal{H}^n}&= |\pa \underline{u}(E\setminus \cup_{i\in I} B_\va(P_i))|_{\mathcal{H}^n} +
\sum_{i\in I}|\pa \underline{u}(B_\va(P_i)\cap E)|_{\mathcal{H}^n} \\&
\geq|E\setminus \cup_{i\in I} B_\va(P_i)|_{\mathcal{H}^n}+ \sum_{i\in I}a_i.
\end{split}
\]
Sending $\va \to 0$, we have $|\pa \underline{u}(E)|_{\mathcal{H}^n}\geq |E|+\sum_{i\in I}a_i$. By the arbitrary choice of $E$ and the definition of Alexandrov solution, we verified
\[
\det \nabla^2 \underline{u}\geq 1+\sum_ia_i\delta_{P_i} \quad \mbox{in }\R^n.
\]
On the other hand, by the comparison principle, we have
\[
\underline{u}(x)\le \frac12 |x|^2.
\]

(iii) Choosing $R$ large such that $\{P_1,\cdots, P_k\}\subset B_R(0)$. Let $u_m$, $m=1,2,3,\cdots$, be the convex generalized solution of
\[
\begin{cases}
\begin{aligned}
\det\nabla^2 u_m&=1+\sum_ia_i\delta_{P_i}& \quad &\mbox{in }B_{R+m},\\
u_m&=\frac12 (R+m)^2& \quad &\mbox{on }\pa B_{R+m}.
\end{aligned}
\end{cases}
\]
By the comparison principle, we have
\[
\underline{u}(x)\le u_m\le \frac12 |x|^2 \quad \mbox{in } B_{R+m}.
\]
Since $u_m$ is convex, after passing a subsequence, $u_m$ locally uniformly converges to some convex function $u$ in $\R^n$. Thus, $u$ satisfies \eqref{eq:more pts} and \eqref{eq:prescrib behavior}.
\end{proof}

\begin{thm} \label{them:elementary} (i) Suppose that $P_1,\cdots, P_k\in \R^n$, $k\geq 1$, $n\ge 2$ and $u$ is a generalized solution of \eqref{eq:many singularity1}. Then $u$ can be uniquely extended to be a convex function in $\R^n$ and satisfies \eqref{eq:more pts},
where $a_i=|\pa u(P_i)|_{\mathcal{H}^n}$. Furthermore, there exists $A\in\mathcal{A}$ and a linear function $\ell(x)$ such that
\be \label{eq:b}
\begin{split}
\limsup_{|x|\to \infty}|x|^{n-2}|u(x)-(\frac12 x^TAx+\ell(x))|<\infty& \quad \mbox{if }n\geq 3,\\
\limsup_{|x|\to \infty}|x||u(x)-(\frac12 x^TAx+d\log\sqrt{x^TAx}+\ell(x))|<\infty& \quad \mbox{if }n=2,
\end{split}
\ee
where 
\be\label{value of d}
d=\frac{1}{2\pi}\sum_{i=1}^k a_i.
\ee

(ii) Conversely, given $a_1,\cdots a_k\in [0,\infty)$, $P_1,\cdots, P_k\in \R^n$, $A, \ell$ as above, there exists a unique generalized solution $u$ of \eqref{eq:more pts} with the asymptotical behavior \eqref{eq:b}.
\end{thm}
\begin{proof}
When $n=2$, the above theorem has been proved in \cite{GMM}. We shall prove the case $n\ge 3$. The first part follows from Proposition \ref{prop:sub} and Theorem \ref{lem:*}. The proof of \eqref{value of d} is the same as that of (1.9) in \cite{CL} and we omit it here.
The second part follows easily from Proposition \ref{prop:to the main result}.
\end{proof}

Let $\mathscr{C}_k$ be the set of all generalized solutions of \eqref{eq:many singularity1} with $k$ distinct non-removable singular points, and $\mathscr{C}_k'$ be the set $\mathscr{C}_k$ modulo the affine equivalence. 
\begin{cor}\label{cor:multip point} Let $k\ge 2$ be an integer. Then for every $ u\in \mathscr{C}_k$, there exists a generalized solution $\tilde{u}$ of
\be\label{eq:c}
\det \nabla^2 \tilde{u}=1+\delta_0+ \sum_{i=1}^{k-1} a_i\delta_{\tilde P_i} \quad \mbox{in }\R^n,
\ee
with the behavior
\be\label{eq:d}
\begin{split}
\limsup_{|x|\to \infty}|x|^{n-2}|\tilde u(x)-\frac12 |x|^2|<\infty\\
\end{split}
\ee
such that $u$ is affine equivalent to $\tilde{u}$, where $a_1,\cdots, a_{k-1}\in (0,\infty)$ and $\tilde P_1, \cdots, \tilde P_{k-1}$ are some distinct points in $\R^n$ with $\tilde P_i\neq 0, i=1,\cdots, k-1$. 

Consequently, $\mathscr{C}_k'$ equals to the set of all solutions of \eqref{eq:c} and \eqref{eq:d} modulo the orthogonal group $O(n)$ and the symmetric group $S_{k-1}$, which can be identified as an orbifold of dimension $d_{n,k}$ , where $d_{n,k}$ is given in \eqref{eq:orbifold dim}.
\end{cor}
\begin{proof}[Proof of Corollary \ref{cor:multip point}] 
 By Theorem \ref{lem:*}, $u$ is affine equivalent to some $\bar{u}\in \mathscr{C}_k$ with asymptotical behavior \eqref{eq:prescrib behavior}. By Proposition \ref{prop:sub}, there exist positive numbers $\bar{a}_1,\cdots, \bar{a}_k$, and $\bar{P}_1,\cdots,\bar{P}_k$ in $\R^n$ such that
\[
\det \nabla^{2}\bar{u}=1+\sum_{i=1}^k \bar{a}_i\delta_{\bar{P}_i} \quad \mbox{in }\R^n.
\]
By some translation and subtracting a linear function, we may assume that $\bar P_k=0$. Let $\tilde{u}(x)=\bar a_k^{-2/n}\bar{u}(\bar a_k^{1/n} x)$. It satisfies
\[
\det\nabla^2 \tilde{u}=1+\delta_0+ \bar a_k^{-1} \sum_{i=1}^{k-1} \bar{a}_i\delta_{\tilde{P}_i}\quad \mbox{in }\R^n,
\]
and \eqref{eq:d}, which proves the first part of this corollary. 

Consequently, $\mathscr{C}_k'$ equals to the set of all solutions of \eqref{eq:c} and \eqref{eq:d} modulo the orthogonal group $O(n)$ and the symmetric group $S_{k-1}$. If we denote \[\mbox{conf}(m,n):=\{(P_1,\cdots,P_m)\in \R^{mn}: P_i\in\R^n\mbox{ and } P_i\neq P_j\mbox{ for } i,j=1,\cdots,m, i\neq j\}\] and $(\R^+)^{k-1}=\{(x_1,\cdots,x_{k-1})\in\R^{k-1}: x_l>0\mbox{ for all }l=1,\cdots, k-1\}$,
then $\mathscr{C}_k'$ can be identified as $\Big((\R^+)^{k-1}\times \big(\mbox{conf}(k-1,n)/O(n)\big)\Big)/S_{k-1}$, which is an orbifold of dimension $d_{n,k}$ given by
\[
d_{n,k}=
\begin{cases}
k-1+\frac{(k-1)k}{2},\quad\mbox{if}\quad k-1\le n,\\
k-1+(k-1)n-\frac{n(n-1)}{2},\quad\mbox{if}\quad k-1> n.
\end{cases}
\]
\end{proof}

\subsection{A strict convexity property}\label{sec:Regularity of local solutions}
We start with a lemma. For $x\in \R^n$, we write $x=(x',x_n)$ with $x'\in \R^{n-1}$.

\begin{lem}\label{lem:strict convex}
Let $\om\subset\R^n$ be a bounded and convex domain with the origin $0\in\pa\om$, $0<\lda<\infty$, $\varphi\in C(\pa\om)$ satisfying $\varphi\ge 0$ on $\pa\om$. Suppose $u$ is a nonnegative generalized solution of 
\[
\det\nabla^2 u\ge\lda \quad\mbox{in } \om
\]
 satisfying $u=\varphi$ on $\pa\om$. If $\varphi(x)\le c|x|^{1+\beta}$ near the origin for some $c>0$ and $\beta>1-\frac 2n$, then $u>0$ in $\om$. 
\end{lem}
\begin{proof}
We argue by contradiction.  By some affine transformation, we may assume that for $e_n=(0',1)$, $2e_n\in\om$ and $u(2e_n)=0$. By the convexity of $u$ and that $u(0)=0$, $u(0', x_n)\equiv 0$ for $x_n\in (0,2)$. For all $|x'|\le \delta$ with sufficiently small $\delta$, it follows from the convexity of $u$ that
\[
u(x', 1)\le 0+ C_1 \varphi(z)\le C_1\cdot c |z|^{1+\beta} \le C_2 \cdot c|z'|^{1+\beta}\le C_3 \cdot c|x'|^{1+\beta},
\]
where $z$ is the intersection point of the ray $2e_n\to (x',1)$ and $\pa \om$, $C_1, C_2, C_3$ depend only on $\pa\om$, and we have used the fact $\pa\om$ is Lipschitz in the above inequalities. 
Using the convexity of $u$ again, we have for all $(x',x_n)$ with sufficiently small $|x'|$, $x_n\in (0,1)$,
\[
u(x', t)\le C |x'|^{1+\beta},
\]
where $C$ only depends on $c$ and $\pa\om$. Consider the cone generated by $\mathcal{C}_r:= B_r'(e_n)$ and $0$ for $r$ small, where $B_r'(e_n)=\{(x',1):|x'|\le r\}$. It is easy to see that the ellipsoid $E_r:=\{4|x'|^2/r^2+(x_n-3/4)^2\leq 1/16\}\subset \mathcal{C}_r$. Let
\[
v=\frac{\lambda^{\frac{1}{n}}r^{2-\frac{2}{n}}}{2\cdot 4^{\frac{n-1}{n}}}\,(4|x'|^2/r^2+(x_n-3/4)^2- 1/16).
\]
We see that $\det \nabla^2 u\ge\lambda =\det \nabla^2 v$. By comparison principle, we have
\[
-\max_{E_r}u\le u-\max_{E_r}u\le v\quad \mbox{in }E_r.
\]
At $x=3e_n/4$, we have
\[
\frac{\lambda^{\frac{1}{n}}r^{2-\frac{2}{n}}}{32\cdot 4^{\frac{n-1}{n}}}\le \max_{E_r}u\le Cr^{1+\beta}.
\]
Sending $r\to 0$, we will obtain a contradiction if $\beta>1-2/n$.
\end{proof}

\begin{cor}\label{ptws-convex}
Let $\om$ be a bounded and convex domain with $0\in\om$, $\om^+=\om\cap \R^n_+$, $\pa^+\om=\pa\om\cap\R^n_+$. Let $u$ be a convex generalized solution of
\[
\begin{split}
\begin{aligned}
\lambda\le \det& \nabla^2 u\le \Lambda&\quad&\mbox{in } \om,\\
u&=f & \quad&\mbox{on }\pa \om,
\end{aligned}
\end{split}
\]
where $0<\lambda\le\Lambda<\infty$, $f\in C(\pa \om)\cap C^{1,\beta}(\pa^+\om)$ with $\beta >1-\frac{2}{n}$. Then $u$ is strictly convex in $\om^+.$
\end{cor}
\begin{proof} 
Suppose $u$ is not strictly convex in $\om^+$ and let $y^1, y^2\in\om^+$ be such that the segment $\overline{PQ}$ is contained in the graph of $u$ with $P=(y^1,u(y^1))$ and $Q=(y^2,u(y^2))$. Let $\ell$ be a supporting hyperplane to $u$ at $(y^1+y^2)/2$ and let $E:=\{z\in\om^+:u(z)=\ell(z)\}$. It follows from Theorem 1 in \cite{Caffarelli90} that $E^*\subset\pa\om^+$, where $E^*$ is the set of extremal points of $E$. Since the line segment $\overline{y^1y^2}\subset E$, we have $\pa^+\om\cap E\neq \emptyset$. Let $z\in \pa^+\om\cap E$ and $v(x)=u(x+z)-\ell(x+z)$. It follows from Lemma \ref{lem:strict convex} that $v(y^1-z)>0$, which contradicts that $u(y^1)=\ell(y^1)$. 
\end{proof}
It is known that if $f\in C^{1,\beta}(\pa \om)$ for $\beta >1-\frac{2}{n}$ then $u$ is strictly convex in $\om$. Lemma \ref{lem:strict convex} and Corollary \ref{ptws-convex} assert that if $f$ is $C^{1,\beta}$ on a portion of the boundary $\pa\om$ with $\beta >1-\frac{2}{n}$, then $u$ is strictly convex in a corresponding portion of $\om$.

\begin{proof}[Proof of Theorem \ref{thm:strict convexity}]
It follows from Corollary \ref{ptws-convex} that $u$ is locally strictly convex in $\om\setminus\mathcal{C}(\Gamma)$. Now, let us discuss the case when $n=3$ or $4$. Suppose $u$ is not locally strictly convex in $\om\setminus\mathcal{L}(\Gamma)$ and let $y^1, y^2\in\om\setminus\mathcal{L}(\Gamma)$ be such that the segment $\overline{PQ}$ is contained in the graph of $u$ with $P=(y^1,u(y^1))$ and $Q=(y^2,u(y^2))$. Let $\ell$ be a supporting hyperplane to $u$ at $(y^1+y^2)/2$ and let $E:=\{z\in\om\setminus\mathcal{L}(\Gamma):u(z)=\ell(z)\}$. First of all, $u$ is convex in $\om$ by Corollary \ref{cor:Yconvex}. Secondly, it follows from the Theorem in \cite{Caffarelli93} that $u(x)\neq \ell(x)$ for those $x\in\om$ not on the line $L$ containing $\overline{y^1y^2}$. Hence, $E^*\subset L$. Since $\overline{y^1y^2}\subset\om\setminus\mathcal{L}(\Gamma)$, there exists $y^3\in\pa\om$ such that $y^1,y^2,y^3$ lie on the same line $L$ and $\overline{y^2y^3}\subset\om\setminus\mathcal{L}(\Gamma)$. It follows from Theorem 1 in \cite{Caffarelli90} that $y^3\in E^*$. This contradicts with Corollary \ref{ptws-convex}.
\end{proof}

\begin{proof}[Proof of Theorem \ref{them:elementary1}]
The first part of Theorem \ref{them:elementary1} follows from Corollary \ref{cor:multip point}. We now prove the regularity part. First, we know from \cite{CL} that for a solution $u$ of \eqref{eq:many singularity1}, it is smooth outside of a large ball. Then, it follows from Theorem \ref{thm:strict convexity} that $u$ is locally strictly convex, and thus, smooth away from the set of line segments each of which connects two singular points.
\end{proof}




\section{Line singularity}\label{sec:line and Y}

\subsection{Existence of global solutions with measure data}\label{sec:measure data}

In this section, we shall prove Theorem \ref{existenceY}.

\begin{thm}\label{existence1}
Let $f\in L^1_{loc}(\R^n)$, $f\ge 0$ in $\R^n$, and the support of $(f-1)$ be bounded. Then for every $c\in \R, b\in\R^n, A\in\mathcal{A}$, there exists a unique convex Alexandrov solution of
\be\label{eq:with f}
\det \nabla^2 u=f\quad\mbox{in }\R^n
\ee
satisfying
\[\lim_{|x|\to +\infty}|E(x)|=0,\]
where $E(x)=u(x)-(\frac 12 x^TAx+b\cdot x+c)$.
\end{thm}

\begin{proof}The uniqueness part follows from the comparison principle, and we will show the existence part.

The proof is similar to that of Theorem 1.7 in \cite{CL}. The difference is that we need to find a proper subsolution so that the estimates depend only on the $L^1$ norm of $f$ instead of the lower bound and $L^{\infty}$ norm of $f$. 

By affine invariance of the equation, we may assume $A=Id, b=0, c=0$. We may also assume $(f-1)$ is supported in $B_{1/2}$. For $R>0$, let $u_R$ be the unique convex Alexandrov solution of
\[
\begin{cases}
\begin{aligned}
\det\nabla^2 u_R&=f&\quad&\mbox{in }B_R,\\
u_R&=R^2/2&\quad&\mbox{on }\pa B_R.
\end{aligned}
\end{cases}
\]
We will show that along a sequence $R\to +\infty$, $u_R$ converges to a solution $u$ of \eqref{eq:with f} satisfying
\be\label{eq:asymp}
\sup_{\R^n}\left|u(x)-\frac{|x|^2}{2}\right|\le C,
\ee
where $C$ depends only on $n$ and $\int_{B_1}f\ud x$. In the following, we may assume that $f$ is smooth as long as the constants in our estimates depends only on $n$ and $\int_{B_1}f\ud x$, since otherwise we can use mollifiers to smooth $f$ and take the limit in the end. We may also suppose $f$ is positive in $B_{1/2}$, since otherwise we replace $f$ by $f+\va\chi$ with a smooth cut-off function $\chi$ which is supported in $B_1$ and equals to $1$ in $B_{1/2}$ and send $\va\to 0$ in the end.

Let $\eta$ be a nonnegative smooth function supported in $B_{1/4}$ satisfying $\int_{B_1}\eta\ud x=1$, and $v_1$ be the smooth solution of
\[
\begin{cases}
\begin{aligned}
\det\nabla^2 v_1&=f+a\eta&\quad&\mbox{in }B_1,\\
v_1&=0&\quad&\mbox{on }\pa B_1,
\end{aligned}
\end{cases}
\]
where $a>0$ will be chosen later. It follows from Alexandrov's maximum principle (see \cite{G}) that
\[
v_1\ge -c(n)\left(\int_{B_1}f(x)\ud x+a\right)^{\frac 1n}=:-c_0\quad\mbox{in }B_{1/2}.
\]
Let $r=|x|, K_1=\frac{4c_0}{3}, K_2=(2K_1)^n, v_2=K_1(r^2-1)$ and
\[
\underline{u}(x)=
\begin{cases}
\int_1^{r}(\tau^n+K_2)^{\frac 1n}\ud \tau,&\quad r\ge 1,\\
v_1, &\quad 0\le r<1.
\end{cases}
\]
First of all, $v_1\ge v_2$ in $\overline{B_{1/2}}$. Secondly, we can choose $a$ large such that $\det\nabla^2 v_2=K_2\ge 1$. Hence, $\det\nabla^2 v_1\le \det\nabla^2 v_2$ in $B_1\setminus \overline{B_{1/2}}$, and it follows from comparison principle that $v_1\ge v_2$ in $B_1\setminus \overline{B_{1/2}}$. So $v_1\ge v_2$ in $\overline{B_{1}}$. Then $\underline u\in C^0(\R^n)\cap C^{\infty}(\overline B_1)\cap C^{\infty}(\overline{\R^n\setminus B_1})$, $\underline u$ is locally convex in $\R^n\setminus B_1$,
\[
\begin{split}
\det\nabla^2 \underline u&=1\quad\mbox{on }\R^n\setminus\overline B_1,\\
\det\nabla^2 \underline u&\ge f\quad\mbox{on } B_1.
\end{split}
\]
Moreover, we have 
\be\label{eq:normal convex}
\underline u\ge v_2\ \ \mbox{in}\ \ B_1, \ \underline u=v_2\ \ \mbox{on}\ \ \pa B_1, \mbox{ and }\lim_{r\to 1^-}|\pa_r  v_2|<\lim_{r\to 1^+}|\pa_r \underline u|.
\ee
Also, since $n\ge 3$, we have
\[
\sup_{\R^n}\left|\underline u(x)-\frac{|x|^2}{2}\right|<+\infty.
\]
Define
\[
\bar{u}(x)=
\begin{cases}
\int_1^{|x|}(\tau^n-1)^{1/n}\ud\tau,&\quad |x|>1,\\
0,&\quad |x|\le 1.
\end{cases}
\]
It follows that
\[
\sup_{\R^n}\left|\bar u(x)-\frac{|x|^2}{2}\right|<+\infty.
\]
Hence
\[
\beta_+:=\sup_{\R^n}\left(\frac{|x|^2}{2}- \bar u(x)\right)<+\infty\quad\mbox{ and }\quad
\beta_-:=\inf_{\R^n}\left(\frac{|x|^2}{2}-\underline u(x)\right)>-\infty.
\]
As Lemma 4.1 in \cite{CL}, we shall show that
\be\label{eq:upperlowerbound}
\underline u(x)+\beta_-\le u_R(x)\le \bar u(x)+\beta_+\quad\forall x\in B_R.
\ee
Indeed, the second inequality of \eqref{eq:upperlowerbound} follows from Lemma 4.1 in \cite{CL}, since our choice of $\bar u$ is the same as the one in \cite{CL}. The first inequality of \eqref{eq:upperlowerbound} follows from the proof of Lemma 4.1 in \cite{CL}, and we include it here for completeness. It is clear that for $\beta$ very negative, we have
\[
\underline u(x)+\beta\le u_R(x)\quad\forall x\in B_R.
\]
Let $\bar \beta$ be the largest number for which the above inequality holds with $\beta=\bar\beta$.
If $\bar\beta\ge\beta_-$, we are done. Otherwise, $\bar\beta<\beta_-$, and for some $\bar x\in \overline B_R$,
\[
\underline u(\bar x)+\bar \beta\le u_R(\bar x).
\]
In view of the boundary data of $u_R$ and the definition of $\beta_-$, we have $|\bar x|<R$. Since
\[
\det \nabla^2\underline u\ge \det\nabla^2 u_R\quad\mbox{in }B_R\setminus\overline B_1
\]
and
\[
\det \nabla^2\underline u\ge \det\nabla^2 u_R\quad\mbox{in } B_1,
\]
we have, by the maximum principle, $|\bar x|=1$. But this is impossible in view of \eqref{eq:normal convex} and the smoothness of $u_R$. 
Hence, the first inequality of \eqref{eq:upperlowerbound} holds. 

Consequently, it follows from the convexity of $u_R$ that $|\nabla u_R|$ is uniformly bounded on every compact subset of $B_{R-1}$. Thus, along a sequence $R_i\to+\infty$,
\[
u_{R_i}\to u\quad \mbox{in } C^{0}_{loc}(\R^n)
\]
for some convex function $u$. Hence $u$ is an Alexandrov solution of \eqref{eq:with f} and satisfies \eqref{eq:asymp}. Finally, it follows from Theorem \ref{lem:*} (i) and \eqref{eq:asymp} that there exists $\tilde c\in \R$ such that 
\[
\lim_{|x|\to\infty}\left|u-\frac{|x|^2}{2}-\tilde c\right|=0.
\]
Thus, $u-\tilde c$ is the desired solution.
\end{proof}

The proof of Theorem \ref{existenceY} follows from a standard approximation method.
\begin{proof}[Proof of Theorem \ref{existenceY}]
Suppose $\mu-1$ is supported in $B_r$. Let $\{f_i\}$ be nonnegative $L^1_{loc}(\R^n)$ functions with supp$(f_i-1)\subset B_{r+1}$ such that
\[
f_i \rightharpoonup \mu
\]
weakly in $B_{r+2}$ and $\int_{B_{r+2}}f_i(x)\ud x\le C$ for some $C$ depending only on $n$ and $\mu$. Let $u_i$ be the solution of \eqref{eq:with f} with $f_i$ instead of $f$ as in Theorem \ref{existence1}. From the above we know that
$|u_i(x)-(\frac 12 x^TAx+b\cdot x+c)|\leq C$ for some $C$ depending only on $n$ and $\mu$. Hence $|u_i|+|\nabla u_i|$ are locally uniformly bounded. Passing to a subsequence (still denoted as $\{u_i\}$), $u_i\to u$ in $C_{loc}^0(\R^n)$ for some convex function $u$, which is an Alexandrov solution of \eqref{eq:1+mu} . As in the end of our proof of Theorem \ref{existence1}, there exists $\tilde c\in\R$ such that $u-\tilde c$ is a desired solution. Finally, the uniqueness part follows from the comparison principle.
\end{proof}

\begin{rem}
This method also provides another proof of Proposition \ref{prop:to the main result}.
\end{rem}

\subsection{Regularity and asymptotical behaviors of solutions near the singularity}\label{asymp-line}

In this section we analyze the behaviors of solutions near the isolated singularity and the line singularity. We will show that $|\nabla^2 u(x)|=O(1/dist (x, \Gamma))$ for $x$ away from the singular set $\Gamma$. This is the best we can have, since the solution in Theorem \ref{thm a} is indeed of this rate. Our proof makes use of Pogorelov estimates in a portion of the domain, which has been used before in \cite{TW3,Savin2} for boundary regularity of solutions of Monge-Amp\`ere equations.

\begin{thm}\label{thm:esti-line}
Let $\om\subset\R^n$ be a bounded smooth convex domain with $n\ge 2$ and $\Gamma\subset\subset\om$ be either a point or a straight line segment. 
Let $f\in C^{1,1}(\overline\om)$, $f>0$ in $\overline\om$, $u$ be a convex function in $\overline\om$ and $u\in C^2(\overline\om\setminus \Gamma)\cap C^4(\om\setminus \Gamma)$. If
\[
\det \nabla^2 u=f\quad \mbox{in }\om\setminus \Gamma,
\]
then for $x\in \om\setminus\Gamma$, we have
\be\label{eq:esti-line}
|\nabla^2 u(x)|\leq \frac{C}{dist(x,\Gamma)},
\ee
where $C>0$ depends only on $n$, $diam(\om)$, $\|\nabla \log f\|_{L^{\infty}(\om)}$, $\|\nabla^2 \log f\|_{L^{\infty}(\om)}$, $\|\nabla u\|_{L^{\infty}(\om)}$ and $\|\nabla^2 u\|_{L^{\infty}(\pa\om)}$, but is independent of $x$.
\end{thm}

\begin{proof} By some translation and rotation, we may suppose that $\Gamma$ lies on the $x_1$-axis. We shall use Pogorelov type estimates.
For $\va>0$, let \[\om_{n,\va}=\{x\in \overline{\om}:x_n\ge \va\}.\]
We first show that \[(x_n-\va)u_{ii}(x)\le C\] for $x\in\om_{n,\va}$, where $i=1,\cdots, n-1$ and $C$ only depends on $n$, $diam(\om)$, $\|\nabla \log f\|_{L^{\infty}(\om)}$, $\|\nabla^2 \log f\|_{L^{\infty}(\om)}$, $\|\nabla u\|_{L^{\infty}(\om)}$ and $\|\nabla^2 u\|_{L^{\infty}(\pa\om)}$. Let
\[
U(x)=(x_n-\va)u_{11}e^{\frac 12 u_1^2}.
\]
If $U$ attains its maximum on $\pa\om_{n,\va}$, we are done. Suppose $U$ attains its maximum at an interior point $x_0$ in $\om_{n,\va}$. By the linear transformation:
\[
\begin{split}
y_i&=x_i,\quad i=2,\cdots, n,\\
y_1&=x_1-\sum_{i=2}^n\frac{u_{1i}(x_0)}{u_{ii}(x_0)}x_i,
\end{split}
\]
which leaves $U$, the equation and $\|\pa_1 f\|_{L^{\infty}}, \|\pa_{11}f\|_{L^{\infty}}$ unchanged (note that later we will only differentiate the equation with respect to $x_1$ twice), we may assume that $u_{1i}(x_0)=0$ for $i=2,\cdots, n$. Let $O$ be an orthogonal rotation which fixes $x_1$ such that $O^t\nabla^2 u(x_0)O$ is diagonal. Let $v(x)=u(Ox)$ and
\[
V(x)=\rho(x) v_{11}e^{\frac 12 v_1^2},
\]
where $\rho(x)=e_n^TOx-\va$ with $e_n=(0,\cdots,0,1)$. Then $V$ achieves its maximum at $\bar x_0=O^tx_0$ in $O^t(\om_{n,\va})$ and $\nabla^2v(\bar x_0)$ is diagonal. Thus, we have, at $\bar x_0$,
\[
\begin{split}
\frac{\rho_i}{\rho}+\frac{v_{11i}}{v_{11}}+v_1v_{1i}&=0,\\
-\frac{\rho_i^2}{\rho^2 }+\frac{v_{11ii}v_{11}-v^2_{11i}}{v^2_{11}}+v_{1i}^2+v_1v_{1ii}&\le 0,
\end{split}
\]
where we have used that $\rho$ is a linear function. Let $L$ be the linear operator at $x_0$,
\[
L:=\sum_{i=1}^n\frac{\pa_{ii}}{v_{ii}}.
\]
Since $\det\nabla^2 v=f(Ox)$, we have
\[
L(v_1)=\pa_1 \log f\quad\mbox{and}\quad L(v_{11})=\sum_{k,l=1}^n\frac{v_{1kl}^2}{v_{kk}v_{ll}}+\pa_{11}\log f.
\]
Consequently, at $\bar x_0$, we have
\[
\begin{split}
0&\ge\sum_{i=1}^n -\frac{\rho_i^2}{\rho^2 v_{ii}}+\frac{v_{11ii}v_{11}-v^2_{11i}}{v^2_{11}v_{ii}}+\frac{v_{1i}^2}{v_{ii}}+\frac{v_1v_{1ii}}{v_{ii}}\\
&=\sum_{i=1}^n -\frac{\rho_i^2}{\rho^2 v_{ii}} +\sum_{k,l=1}^n\frac{v_{1kl}^2}{v_{11}v_{kk}v_{ll}}+\frac{\pa_{11}\log f}{v_{11}}-\sum_{i=1}^n\frac{v^2_{11i}}{v^2_{11}v_{ii}}+v_{11}+v_1\pa_1\log f\\
&\ge \sum_{i=1}^n -\frac{\rho_i^2}{\rho^2 v_{ii}} +\sum_{i=2}^n\frac{v^2_{11i}}{v^2_{11}v_{ii}}+v_{11}+\frac{\pa_{11}\log f}{v_{11}}+v_1\pa_1\log f\\
&\ge \sum_{i=1}^n -\frac{\rho_i^2}{\rho^2 v_{ii}}+\sum_{i=2}^n \frac{1}{v_{ii}}\left(\frac{\rho_i}{\rho}\right)^2 +v_{11}+\frac{\pa_{11}\log f}{v_{11}}+v_1\pa_1\log f\\
&=v_{11}+\frac{\pa_{11}\log f}{v_{11}}+v_1\pa_1\log f,
\end{split}
\]
where we used that $\rho_1=0$. Hence $v_{11}\le C$, and thus, $(x_n-\va)u_{11}\le C$ in $\om_{n,\va}$, where $C$ depends only on $n$, $diam(\om)$, $\|\nabla \log f\|_{L^{\infty}(\om)}$, $\|\nabla^2 \log f\|_{L^{\infty}(\om)}$, $\|\nabla u\|_{L^{\infty}(\om)}$ and $\|\nabla^2 u\|_{L^{\infty}(\pa\om)}$. Similarly, we can show that $(x_n-\va)u_{ii}\le C$ in $\om_{n,\va}$ for $i=1,\cdots, n-1$.

Next, we shall show that \[(x_n-\va)u_{nn}\le C\] in $\om_{n,\va}$. Let
\[
W(x)=(x_n-\va)u_{nn}e^{\frac 12 u_n^2}.
\]
If $W$ attains its maximum on $\pa\om_{n,\va}$, we are done. Suppose $W$ attains its maximum at an interior point $x_0$ in $\om_{n,\va}$. Let $T$ be the linear transformation
\[
\begin{split}
y_i&=x_i,\quad i=1,\cdots, n-1,\\
y_n&=x_n-\sum_{i=1}^{n-1}\frac{u_{in}(x_0)}{u_{nn}(x_0)}x_i,
\end{split}
\]
and $v(x)=u(Tx)$. Let
\[
\tilde W=(x_n-\sum_{i=1}^{n-1}\frac{u_{in}(x_0)}{u_{nn}(x_0)}x_i -\va)v_{nn}e^{\frac 12 v_n^2},
\]
for $x\in T^{-1}(\om_{n,\va})$. Then $\tilde W$ attains its maximum at $\tilde x_0=T^{-1}(x_0)$ in $T^{-1}(\om_{n,\va})$ and $v_{in}(\tilde x_0)=0$ for $i=1,\cdots, n-1$. Let $O=(O_{ij})$ be an orthogonal rotation which fixes $x_n$ such that $O^t\nabla^2 v(\tilde x_0)O$ is diagonal. Let $w(x)=v(Ox)$ and
\[
\overline W=\rho(x)v_{nn}e^{\frac 12 v_n^2},
\]
where $\rho(x)=(x_n-\sum_{i,j=1}^{n-1}\frac{u_{in}(x_0)}{u_{nn}(x_0)}O_{ij}x_j -\va)$. Then $\overline W$ achieves its maximum at $\bar x_0= O^{-1}T^{-1}(x_0)$ in $ O^{-1}T^{-1}(\om_{n,\va})$. By the same arguments as above, we obtain, at $\bar x_0$,
\[
0\ge -\frac{\rho_n^2}{\rho^2 w_{nn}}+w_{nn}+\frac{\pa_{nn}\log f}{w_{nn}}+w_n\pa_n\log f.
\]
Hence, $\rho(\bar x_0) v_{nn}(\bar x_0)\le C$, and thus, $W\le C$, where $C>0$ depends only on $n$, $diam(\om)$, $\|\nabla \log f\|_{L^{\infty}(\om)}$, $\|\nabla^2 \log f\|_{L^{\infty}(\om)}$, $\|\nabla u\|_{L^{\infty}(\om)}$ and $\|\nabla^2 u\|_{L^{\infty}(\pa\om)}$. So we can conclude that $(x_n-\va)u_{nn}\le C$ in $\om_{n,\va}$.

By sending $\va\to 0$, we have that for all $x\in \om$ with $x_n>0$,
\[
|\nabla^2 u(x)|\le C/x_n.
\]
Similarly, we can show that for all $x\in \om\setminus\overline\Gamma$, we have
\[
|\nabla^2 u(x)|\le \frac{C}{dist(x,\Gamma)},
\]
where $C>0$ depends only on $n$, $diam(\om)$, $\|\nabla \log f\|_{L^{\infty}(\om)}$, $\|\nabla^2 \log f\|_{L^{\infty}(\om)}$, $\|\nabla u\|_{L^{\infty}(\om)}$ and $\|\nabla^2 u\|_{L^{\infty}(\pa\om)}$, but is independent of $x$.
\end{proof}

\begin{proof}[Proof of Theorem \ref{thm:asymptotical behavior}]
It is clear $u$ is locally strictly convex, and thus, smooth away from $\Gamma$. Hence, Theorem \ref{thm:asymptotical behavior} follows from Theorem \ref{thm:esti-line}.
\end{proof}

Finally, we study the Dirichlet problem with isolated and line singularities and show some explicit dependence of the constant $C$ in \eqref{eq:esti-line} from the give data.
\begin{thm} \label{thm:line}
Let $\om$ be a bounded strictly convex domain in $\R^n$ with $\pa \om \in C^4$ and $n\ge 2$, $f\in C^{1,1}(\overline \om)$, $f>0$ in $\overline \om$ and $\varphi\in C^4(\pa \om)$, $\Gamma\subset\subset\om$ be either a point or a straight line segment and $\mu$ be a finite Borel measure supported on $\Gamma$.  Let $u\in C(\overline \om)$ be the unique generalized convex solution of the Dirichlet problem
\be \label{eq:mae-line}
\begin{split}
\begin{aligned}
\det \nabla^2 u&=f+\mu&\quad& \mbox{in }\om,\\
u&=\varphi& \quad & \mbox{on } \pa \om.
\end{aligned}
\end{split}
\ee
Then $u\in C^{0,1}(\om) \cap C^{3,\al}_{loc}(\overline \om \setminus \Gamma )$ for any $\al\in (0,1)$. Moreover, we have
\be\label{eq:mae-line-esti}
|\nabla^2 u(x)|\leq \frac{C}{dist(x,\Gamma)},
\ee
where $C>0$ depends only on $\om, n, \min_{\overline \om}f,\|f\|_{C^{1,1}(\overline \om)}, \|\varphi\|_{C^4(\pa \om)}, \mu(\om)$, and $ dist(\Gamma,\pa \om)$.
\end{thm}
\begin{proof} We divide the proof into several steps.

\medskip

\emph{Step 1.} $C^0$ estimate.

\medskip

Since $u$ is convex, $u\leq \max_{\pa \om}\varphi$. If $u\geq \min_{\pa \om}\varphi$, we are done. Otherwise, let
$D=\{u<\min_{\pa \om}\varphi\}\subset \om$. By the Alexandrov's maximum principle, we have
\[
|u(x)-\min_{\pa \om}\varphi|^n\leq C(n) diam(D)^{n-1} dist(x,\pa D)\Big(\int_{D} f\,\ud x+\mu(\om)\Big).
\]
Hence $|u|\leq C_0$ for some positive $C_0$ depending only on $\om, n, \mu(\om), \|\varphi \|_{L^\infty(\pa \om)}$ and $\|f \|_{L^\infty( \om)}$.

\medskip

\emph{Step 2.} $C^{0,1}$ estimate.

\medskip

Clearly, the harmonic extension $h$ of $\varphi$ in $\om$ provides an upper bound of $u$. We shall construct a function which provides a lower bound of $u$.
Let $u_1,u_2 \in C^{3,\al}(\overline \om)$ be the solutions (see \cite{CNS, TW3}) of
\[
\begin{split}
\det \nabla^2 u_1&=f \quad \mbox{in }\om,\\
u_1&=\varphi \quad  \mbox{on } \pa \om,
\end{split}
\]
and
\[
\begin{split}
\det \nabla^2 u_2&=1 \quad \mbox{in }\om,\\
u_2&=0 \quad  \mbox{on } \pa \om,
\end{split}
\]
respectively. Applying the Alexandrov's maximum principle to $u_2$, we see that there is a constant $A>0$ depending only on
$\om, n, \mu(\om), dist(\Gamma,\pa \om), \|\varphi \|_{L^\infty(\pa \om)}$ and $\|f \|_{L^\infty( \om)}$ such that for $\underline{u}(x):= u_1(x)+Au_2(x)$
\[
\sup_{\Gamma}\underline{u}\leq \inf_{\Gamma}u.
\]
On the other hand, $ \underline{u}=u$ on $\pa \om$ and $\det\nabla^2\underline{u} >f=\det\nabla^2u$ in $\om\setminus\Gamma$.
It follows from the comparison principle that $\underline{u}\leq u$ in $\om$. In conclusion, we have
\[
h=u=\underline{u}\mbox{ on } \pa\om\quad\mbox{ and } \quad  h\geq u \geq \underline{u} \mbox{ in }\om.
\]
Hence, for any $x\in \pa \om$,  $|\pa u(x)|\leq C$. Since $u$ is convex, $diam (\pa u(\om))\leq C_1$ for some
$C_1>0$ depending only on $\om, n, \mu(\om), dist(\Gamma,\pa \om), \min_{\overline \om}f, \|f\|_{C^{1,1}(\overline \om)}, \|\varphi\|_{C^4(\pa \om)}$.

\medskip

\emph{Step 3.} $C^2$ estimates for approximating solutions $u_\va$ on the boundary $\pa \om$.

\medskip

Let us consider the following approximating problem
\[
\begin{split}
\begin{aligned}
\det \nabla^2 u_\va&=f+\eta_\va(x) &\quad &\mbox{in }\om,\\
u_\va&=\varphi& \quad & \mbox{on } \pa \om,
\end{aligned}
\end{split}
\]
where $\eta_\va$ is nonnegative and smooth, supp$(\eta_\va)\subset \subset Q_\va(\Gamma):=\{x\in\om: dist(x,\Gamma)< \va\}$ and $\eta_\va\rightharpoonup \mu$ weakly as $\va\to 0$. We may also assume $f$ is smooth. Then, up to a subsequence, $u_\va\to u$ in $C^0_{loc}(\om)$.
 Let $\theta=\frac{1}{10}dist(\Gamma,\pa \om)$.  As in step 2, we can construct a subsolution $\underline{u}$ such that
\[
\begin{split}
\begin{aligned}
\det \nabla^2 \underline{u}&\geq f+A& \quad &\mbox{in }\om,\\
\underline{u}&=\varphi& \quad & \mbox{on } \pa \om,
\end{aligned}
\end{split}
\]
and $\underline{u}\leq u-\theta$ in $\pa Q_\theta(\Gamma)$. Hence, we have $\underline{u}\leq u_\va$ on $\pa Q_\theta(\Gamma)$
for small $\va$. By the comparison principle,
\[
\underline{u}\leq u_\va\leq h\quad \mbox{in } \om\setminus Q_\theta(\Gamma)\quad\mbox{for all small }\va.
\]
Hence, $|u_\va|_{C^{0,1}(\om)}$ is uniformly bounded, and thus, $u_\va\to u$ in $C^0(\overline\om)$. Furthermore, since
\[
\begin{split}
\det \nabla^2 u_\va&= f \quad \mbox{in }\om\setminus Q_\theta(\Gamma),\\
u_\va&=\varphi \quad  \mbox{on } \pa \om
\end{split}
\]
for small $\va$, the $C^2$ boundary estimate in Theorem 2.1  of \cite{Guan} gives
\[
|\nabla^2 u_\va |\leq C \quad \mbox{on }\pa\om,
\]
where $C>0$ depends only on $\om, n, \mu(\om), \min_{\overline \om}f,\|f\|_{C^{1,1}(\overline \om)}, \|\varphi\|_{C^4(\pa \om)}$  and $ dist(\Gamma,\pa \om)$.

\medskip

\emph{Step 4.} $C^2$ estimates for $u_\va$ away from $\Gamma$ and complete the proof.

\medskip

It follows from Theorem \ref{thm:esti-line} and the above three steps that for $\tau>0$ we have
\[
|\nabla^2 u_\va(x)| \leq \frac{C}{dist(x, \Gamma)}\quad\forall x\in \om\setminus\overline{Q_\tau(\Gamma)}
\]
if $\va$ is sufficiently small, where $C>0$ depends only on $\om$, $n$, $\mu(\om)$, $\min_{\overline \om}f$, $\|f\|_{C^{1,1}(\overline \om)}$, $\|\varphi\|_{C^4(\pa \om)}$  and $ dist(\Gamma,\pa \om)$. Sending $\va\to 0$ first and then sending $\tau\to 0$, we have
\[
|\nabla^2 u(x)| \leq \frac{C}{dist(x, \Gamma)}\quad\forall x\in \om\setminus\Gamma.
\]

The rest of the theorem follows from Evans-Krylov theorem and Schauder estimates of elliptic equations. In conclusion, we complete the proof.
\end{proof}

\begin{rem}\label{rem:convex set}
In fact, both Theorem \ref{thm:esti-line} and Theorem \ref{thm:line} also hold for $\Gamma$ being a convex set, which follows from the same proofs as above.
\end{rem}

\bigskip

\noindent Tianling Jin

\noindent Department of Mathematics, The University of Chicago\\
5734 S. University Avenue, Chicago, IL, 60637 USA\\[1mm]
Email: \textsf{tj@math.uchicago.edu}

\medskip

\noindent Jingang Xiong

\noindent Beijing International Center for Mathematical Research, Peking University\\
Beijing 100871, China\\[1mm]
Email: \textsf{jxiong@math.pku.edu.cn}

\end{document}